\documentclass[a4paper,12pt,oneside]{amsart}       

\usepackage[british,english]{babel} 
\usepackage{amsmath,amssymb,amsthm,amscd}   
\usepackage{mathrsfs}
\usepackage[all]{xy}
\usepackage{stmaryrd}

\theoremstyle{cited}
\newtheorem{teor}{Theorem}[section]
\newtheorem{lem}[teor]{Lemma}
\newtheorem{cor}[teor]{Corollary}
\newtheorem{prop}[teor]{Proposition}

\theoremstyle{definition}
\newtheorem{deft}[teor]{Definition}

\theoremstyle{remark}
\newtheorem{oss}[teor]{Remark}

\newcommand{\C}{\mathbb{C}}
\newcommand{\R}{\mathbb{R}}
\newcommand{\Z}{\mathbb{Z}}
\newcommand{\N}{\mathbb{N}}

\title{The $\omega$-Borel invariant for representations into $SL(n,\C_\omega)$}

\author{Alessio Savini}

\begin{document}

\maketitle

\begin{abstract}
Let $\Gamma$ be the fundamental group of a complete hyperbolic $3$-manifold $M$ with toric cusps. By following~\cite{iozzi:articolo} we define the $\omega$-Borel invariant $\beta_n^\omega(\rho_\omega)$ associated to a representation $\rho_\omega: \Gamma \rightarrow SL(n,\C_\omega)$, where $\C_\omega$ is a field introduced by~\cite{parreau:articolo} which can be constructed as a quotient of a suitable subset of $\C^\N$ with the data of a non-principal ultrafilter $\omega$ on $\N$ and a real divergent sequence $\lambda_l$ such that $\lambda_l \geq 1$. \\
\indent Since a sequence of $\omega$-bounded representations $\rho_l$ into $SL(n,\C)$ determines a representation $\rho_\omega$ into $SL(n,\C_\omega)$, for $n=2$ we study the relation between the invariant $\beta^\omega_2(\rho_\omega)$ and the sequence of Borel invariants $\beta_2(\rho_l)$. We conclude by showing that if a sequence of representations $\rho_l:\Gamma \rightarrow SL(2,\C)$ induces a representation $\rho_\omega:\Gamma \rightarrow SL(2,\C_\omega)$ which determines a reducible action on the asymptotic cone $C_\omega(\mathbb{H}^3,d/\lambda_l,O)$ with non-trivial length function, then it holds $\beta^\omega_2(\rho_\omega)=0$.
\end{abstract}

\vspace{20pt}


\section{Introduction}

Given a finitely generated group $\Gamma$, the character variety $X(\Gamma,SL(n,\C))$ is an algebraic variety obtained as $GIT$-quotient of the representation variety $R(\Gamma,SL(n,\C))$ by the conjugation action of $SL(n,\C)$. When $\Gamma$ is the fundamental group of a complete hyperbolic $3$-dimensional manifold $M$ with toric cusps, it is possible to attach to every equivalence class of representations a suitable invariant called Borel invariant. Indeed, in~\cite{iozzi:articolo} the authors prove that the Borel class $\beta(n)$, already introduced and studied in~\cite{goncharov:articolo}, is a generator for the cohomology group $H^3_{cb}(PSL(n,\C))$. Thus, given a representation $\rho:\Gamma \rightarrow PSL(n,\C)$, we can construct a class into $H_b^3(\Gamma)$ by pulling back $\beta(n)$ along $\rho_b^*$ and then evaluate this new class on a fundamental class $[N,\partial N] \in H^3(N,\partial N)$. Here $N$ is a compact core of $M$. When $n=2$ this invariant is exactly the volume of the representation defined as the integral of the pullback of the standard volume form $\omega_{\mathbb{H}^3}$ along any pseudo-developing map $D$, as written both in~\cite{dunfield:articolo} and in~\cite{franc04:articolo} (see for instance~\cite{kim:articolo} for a proof of the equivalence). The Borel invariant of a representation $\rho:\Gamma \rightarrow SL(n,\C)$ will be the Borel invariant of the induced representation into $PSL(n,\C)$. Moreover, since this invariant remains unchanged under conjugation, we have a well-defined function on the character variety $X(\Gamma,SL(n,\C))$, called Borel function, which is continuous with respect to the topology of the pointwise convergence.\\
\indent Inspired by the work of Thurston about the compactification of the Teichmuller space for a closed surface of genus $g$ exposed in~\cite{thurston:articolo} and generalizing the constructions for algebraic curves appeared in~\cite{culler:articolo}, in ~\cite{morgan:articolo} J. Morgan and P. Shalen proposed a new way to compactify a generic algebraic variety $V$ given a generating set $\mathcal{F}$ for the algebra of regular functions $\C[V]$. This particular method applied to the character variety $X(\Gamma,SL(2,\C))$ allows to interpret the ideal points of the compactification as projective length functions of isometric $\Gamma$-actions on real trees which are constructed as Bass--Serre trees associated to $SL(2,\mathbb{K}_v)$, where $\mathbb{K}_v$ is a suitable valued field (see~\cite{serre:libro}). A more geometric approach based on Gromov--Hausdorff convergence was suggested by both~\cite{bestvina:articolo} and~\cite{paulin:articolo}. Lately~\cite{parreau:articolo} extended this intepretation to the more general case of $X(\Gamma,SL(n,\C))$ by viewing an ideal point as a projective vectorial length function relative to an isometric action, this time on a Euclidean building of type $A_{n-1}$. The method suggested by~\cite{parreau:articolo} to obtain the Euclidean building and its isometric $\Gamma$-action is based on asymptotic cones and it reminds the ones already exposed both in~\cite{bestvina:articolo} and in~\cite{paulin:articolo}.\\
\indent In the attempt to link all these ideas, one could naturally ask if it is possible to extend continuously the Borel function to the ideal points of the compactification of $X(\Gamma,SL(n,\C))$. Going further, one could be interested in studying the possible values attained at ideal points and trying to formulate a rigidity result, which would generalize~\cite[Theorem 1]{iozzi:articolo}.\\
\indent The aim of this paper is to make a small step towards this direction by defining a numerical invariant, the $\omega$-Borel invariant, associated to a representation $\rho_\omega:\Gamma \rightarrow SL(n,\C_\omega)$, where $\C_\omega$ is a field obtained as a quotient of a suitable subset of $\C^\N$ by an equivalence relation which depends on a non-principal ultrafilter $\omega$ on $\N$ and a real divergent sequence $\lambda_l$ with $\lambda_l \geq 1$. The motivation of this definition relies on the interpretation of the limit action of $\Gamma$ on the Euclidean bulding of type $A_{n-1}$ as a representation $\rho_\omega: \Gamma \rightarrow SL(n,\C_\omega)$, as proved in~\cite[Theorem 5.2]{parreau:articolo}.\\
\indent The first section is dedicated to preliminary definitions, in particular we recall the definition of the field $\C_\omega$ and the notion of bounded cohomology of locally compact groups. In the second section we give the definition of the $\omega$-Borel cohomology class $\beta^\omega(n)$ which will be an element of $H^3_{b}(SL^\delta(n,\C_\omega))$. In the last section we define the $\omega$-Borel invariant $\beta^\omega_n(\rho_\omega)$ for a representation $\rho_\omega:\Gamma \rightarrow SL(n,\C_\omega)$ and we describe some of its properties. In particular we focus our attention on the case $n=2$. We show that given a sequence of representations $\rho_l:\Gamma \rightarrow SL(2,\C)$ which diverges to an ideal point to the character variety and such that the induced representation $\rho_\omega:\Gamma \rightarrow SL(2,\C_\omega)$ determines a reducible action on the asymptotic cone $C_\omega(\mathbb{H}^3,d_{\mathbb{H}^3}/\lambda_l,O)$, then it must holds $\beta^\omega_2(\rho_\omega)=0$.\\
\\
\textbf{Acknowledgements}: I would like to thank both Prof. Alessandra Iozzi and Prof. Marc Burger for the enlightening discussions and the help they gave me during my visiting period at ETH.


\section{Preliminary definitions}

\subsection{The field $\C_\omega$}

For more details regarding the definitions and the results contained in this section we refer to~\cite[Section 3.3]{parreau:articolo}. We start by recalling the notion of ultrafilter and some fundamental properties that we are going to exploit lately.

\begin{deft}
An \textit{ultrafilter} $\omega$ on a set $X$  is a family of subsets of $X$ which satisfies:

\begin{itemize}
	\item The empty set is not contained in $\omega$, that is $\varnothing \notin \omega$.
	\item If $A \subset B$ and $A \in \omega$, then $B \in \omega$.
	\item Given a collection $A_1, \ldots A_n$ such that $A_i \in \omega$ for every $i=1,\ldots,n$, then $A_1 \cap \ldots \cap A_n \in \omega$.
	\item Given $A_1,\ldots A_n$ such that $A_1 \sqcup \ldots \sqcup A_n= X$, there exists exactly one $i_0 \in \{1,\ldots,n\}$ so that $A_{i_0} \in \omega$. 
\end{itemize}

An ultrafilter is \textit{principal} and centered at $x \in X$ if for every set $A \in \omega$ it holds $x \in A$. Otherwise we say that the ultrafilter is \textit{non-principal}.
\end{deft}

The importance of ultrafilters relies on their power to force convergence of sequences of points in a topological space $X$ by selecting a suitable limit point. For the sake of clarity we first need to introduce the following  

\begin{deft}
Let $X$ be a topological space and let $(x_k)_{k \in \N}$ be a sequence of points in $X$. Fix an ultrafilter $\omega$ on the set of natural numbers $\N$. We say that the sequence $\omega$-\textit{converges} to $x_0$ if for every open neighborhood $U$ of $x_0$ we have $\{k \in \N: x_k \in U\} \in \omega$.  
\end{deft}

A priori a sequence may admit no limit or several limits if the topology of the space $X$ does not have good properties. To guarantee the existence and the uniqueness of the limit we need a compact Hausdorff space. Indeed, it holds
\begin{prop}\label{existence}
Let $X$ be a topological space which is compact and Hausdorff. Then, for any ultrafilter $\omega$ on $\N$ and any sequence $(x_k)_{k \in \N}$ of points in $X$, there exists a unique point $x_0 \in X$ such that

\[
\textup{$\omega$-}\lim_{k \to \infty} x_k = x_0.
\]
\end{prop}

Another remarkable property of ultrafilters is the compatibility with continuous functions between topological spaces. 

\begin{prop}\label{continuous}
Let $f:X \rightarrow Y$ be a continuous function between two compact Hausdorff spaces. Let $\omega$ be an ultrafilter on $\N$. For any sequence $(x_k)_{k \in \N}$ of points in $X$ we have

\[
\textup{$\omega$-}\lim_{k \to \infty} f(x_k)=f(\textup{$\omega$-}\lim_{k \to \infty} x_k).
\]
\end{prop}

We are now ready to describe the construction of the field $\C_\omega$. Let $\omega$ be a non-principal ultrafilter on $\N$ and let $(\lambda_k)_{k \in \N}$ be a real sequence that diverges to infinity and such that $\lambda_k \geq 1$ for every $k$. We define

\[
\C_\omega = \{ (a_k) \in \C^\N | \exists C > 0,\forall k \hspace{5pt} |a_k|^\frac{1}{\lambda_k} < C  \}/\sim_\omega
\]
where $(a_k)_{k \in \N} \sim_\omega (b_k)_{k \in \N}$ if and only if $\omega$-$\lim_{k \to \infty} |a_k - b_k|^\frac{1}{\lambda_k}=0$. It is easy to verify that the operations of pointwise sum and pointwise multiplication defined over $\C^\N$ are compatible with the equivalence relation $\sim_\omega$. Thus they define two operations of sum and multiplication over $\C_\omega$, which make $\C_\omega$ a field. 
There is a natural field embedding of $\C$ into $\C_\omega$ given by the constant sequences. \\
\indent If we denote by $a_\omega$ the equivalence class $[(a_k)]$ of the sequence $(a_k)_{k \in \N}$, the function

$$
|a_\omega|^\omega:=\text{$\omega$-$\lim_{k \to \infty} |a_k|^\frac{1}{\lambda_k}$}
$$
is an ultrametric absolute value on $\C_\omega$, that is it satisfies 

\[
|a_\omega + b_\omega|^\omega \leq \max\{|a_\omega|^\omega,|b_\omega|^\omega\}
\] 
for every pair $a_\omega, b_\omega \in \C_\omega$. It is worth noticing the elements of $\C$, seen as the subfield of constant sequences, have all norm equal to $1$. 

\begin{deft}
The ultrametric field $(\C_\omega,|\cdot|^\omega)$ is called the \textit{asymptotic cone} of $(\C,|\cdot|)$ with respect to the scaling sequence $(\lambda_k)_{k \in \N}$ and the ultrafilter $\omega$.
\end{deft}

If we consider the distance induced by the absolute value $|\cdot|^\omega$ and we endow $\C_\omega$ with the metric topology, we obtain a topological field which is complete (see~\cite[Remark 3.10]{parreau:articolo}), but it is not locally compact.

\begin{prop}\label{locally}
The field $\C_\omega$ is not locally compact with respect to the metric topology induced by the absolute value $|\cdot|^\omega$. 
\end{prop}

\begin{proof}
Since $\C_\omega$ is a normed space, local compactness can be checked by verifying the compactness of the unit closed ball. Hence, it suffices to show that the closed ball

\[
\overline{B}_1(0):=\{ a_\omega \in \C_\omega| |a_\omega|^\omega \leq 1 \}
\]
is not compact. We are going to show that it is not sequentially compact. Consider the sequence $(n)_{n \in \N}$ where each element $n$ has to be thought of as an element of $\C_\omega$ thanks to the standard embedding given by constant sequences. Given two different elements $n$ and $m$ it is clear that their distance in $\C_\omega$ is always equal to $1$, indeed

\[
|n-m|^\omega=\textup{$\omega$-}\lim_{k \to \infty} |n-m|^\frac{1}{\lambda_k}=1.
\]

Hence it cannot exist a subsequence of $(n)_{n \in \N}$ which converges, as desired. 
\end{proof}

The construction exposed above can be repeated, rather than for a field, for every $m$-dimensional normed vector space $(V, || \cdot ||)$ over $\C$. More precisely, we define 

\[
V_\omega:=\{ (v_k) \in V^\N| \exists C >0, \forall k \hspace{5pt} ||v_k||^\frac{1}{\lambda_k} < C\}/ \sim_\omega,
\]
where $(v_k)_{k \in \N}$ and $(u_k)_{k \in \N}$ are equivalent if and only if $\omega$-$\lim_{k \to \infty} || u_k - v_k ||^\frac{1}{\lambda_k}=0$. Let $v_\omega$ be the equivalence class determined by $(v_k)_{k \in \N}$. It is possible to endow $V_\omega$ with a structure of $m$-dimensional $\C_\omega$-vector space by considering the operations induced by pointwise sum and by pointwise scalar multiplication. As before, we have a well-defined norm $||\cdot||^\omega$ given by

\[
||v_\omega||^\omega:=\text{$\omega$-$\lim_{k \to \infty} ||v_k||^\frac{1}{\lambda_k}$}.
\]

\begin{deft}
The $\C_\omega$-vector space $(V_\omega,||\cdot||^\omega)$ is the \textit{asymptotic cone} of the vector space $(V,||\cdot||)$ with respect to the scaling sequence $(\lambda_k)_{k \in \N}$ and the ultrafilter $\omega$. 
\end{deft}

We now focus our attention on the set of complex square matrices of order $n$, namely $M(n,\C)$. If we choose as norm over $M(n,\C)$ the standard matrix norm, we can apply the construction above to the normed vector space $(M(n,\C),||\cdot||)$. In this particular case we are able to enrich the structure of $M(n,\C)_\omega$ by considering a multiplication. Indeed, the classic multiplication rows-by-columns is compatible with $\sim_\omega$ and hence it defines a structure of $\C_\omega$-algebra on $M(n,\C)_\omega$.

\begin{deft}
The normed algebra $(M(n,\C)_\omega,||\cdot||^\omega)$ is called the \textit{asymptotic cone} of the algebra $(M(n,\C),||\cdot||)$ with respect to the scaling sequence $(\lambda_k)_{k \in \N}$ and the ultrafilter $\omega$. 
\end{deft}

\begin{deft}
A sequence $(g_k) \in GL(n,\C)^\N$ is $\omega$-\textit{bounded} if

\[
\exists C>0: \forall k \hspace{5pt} ||g_k||^\frac{1}{\lambda_k},||g_k^{-1}||^\frac{1}{\lambda_k} <C.
\]

The previous condition implies that the sequence $(g_k)_{k \in \N}$ defines an element of $M(n,\C)_\omega$ which admits a multiplicative inverse. We denote by $GL(n,\C)_\omega$ the set of all the invertible elements of $M(n,\C)_\omega$. This is a group with respect to the multiplication rows-by-columns. We denote by $SL(n,\C)_\omega$ the subgroup

\[
SL(n,\C)_\omega:=\{ g_\omega \in GL(n,\C)_\omega| \exists (g_k)_{k \in \N} \in g_\omega: \forall k \hspace{5pt} \det(g_k)=1 \}.
\]
\end{deft}

Since we can also consider the normed algebra $(M(n,\C_\omega),||\cdot||_\infty)$, where $||\cdot||_\infty$ is the standard supremum norm with respect to $|\cdot|^\omega$, it is natural to ask whether this algebra is isomorphic to $M(n,\C)_\omega$ as normed algebra. The answer is given by~\cite[Corollary 3.18]{parreau:articolo}, which states that there is a natural isomorphism as normed $\C_\omega$-algebras between $M(n,\C)_\omega$ and $M(n,\C_\omega)$. Moreover this isomorphism induces an isomorphism of groups between $SL(n,\C)_\omega$ and $SL(n,\C_\omega)$. \\
\indent We conclude this section by introducing the space $\mathbb{P}^1(\C)_\omega$. In order to do this, we first need to recall the construction of the asymptotic cone of $\mathbb{H}^3$. 

\begin{deft}

Let $(x_k)_{k \in \N}$ be a sequence of basepoints in $\mathbb{H}^3$. Consider the space
\[
C_\omega(\mathbb{H}^3,d/\lambda_k,x_k):=\{ (y_k) \in (\mathbb{H}^3)^\N|\exists C>0,\forall k \hspace{5pt} d(x_k,y_k)<C \lambda_k\}/\sim_\omega
\]
where $(y_k)_{k \in \N} \sim_\omega (y'_k)_{k \in \N}$ if and only if $\omega$-$\lim_{k \to \infty} d(y_k,y_k')/\lambda_k=0$. Denote by $y_\omega$ the equivalence class of the sequence $(y_k)_{k \in \N}$. If we define 

\[
d_\omega(y_\omega,y'_\omega)=\textup{$\omega$-}\lim_{k \to \infty}d(y_k,y_k')/\lambda_k
\]
we get a metric and the metric space $(C_\omega(\mathbb{H}^3,d/\lambda_k,x_k),d_\omega)$ is the \textit{asymptotic cone} with respect to the ultrafilter $\omega$, the scaling sequence $(\lambda_k)_{k \in \N}$ and the sequence of basepoints $(x_k)_{k \in \N}$.
\end{deft}

Assume to fix the origin $O$ of the Poincar\'e model of $\mathbb{H}^3$ as the constant sequence of basepoints for the asymptotic cone construction. It should be clear that there exists a natural surjection 

\[
\pi :\mathbb{P}^1(\C)^\N \rightarrow \partial_\infty C_\omega(\mathbb{H}^3,d/\lambda_k,O)
\]
defined as it follows. Thinking of $\mathbb{P}^1(\C)$ as the boundary at infinity of $\mathbb{H}^3$, a sequence of points $(\xi_k) \in \mathbb{P}^1(\C)^\N$ determines in a unique way a sequence of geodesic rays $(c_k)_{k \in \N}$ starting from $O$ and ending at $(\xi_k)_{k \in \N}$. These rays allows us to define a geodesic ray $c_\omega:[0,\infty) \rightarrow C_\omega(\mathbb{H}^3,d/\lambda_k,O)$ given by $c_\omega(t):=[c_k(\lambda_k t)]$. Hence, we can define $\pi((\xi_k)_{k \in \N}):=c_\omega(\infty)$. The space $\mathbb{P}^1(\C)_\omega$ will be the quotient of $\mathbb{P}^1(\C)^\N$ by the equivalence relation induced by the surjection $\pi$. In this way $\mathbb{P}^1(\C)_\omega$ is clearly identified with boundary at infinity of $C_\omega(\mathbb{H}^3,d/\lambda_k,O)$ and hence inherits in a natural way an action of $SL(2,\C)_\omega$ given by $[h_k].[\xi_k]:=[h_k.\xi_k]$. This action is well defined because the action of $SL(2,\C)_\omega$ on $C_\omega(\mathbb{H}^3,d/\lambda_k,O)$ is well defined (see~\cite[Proposition 3.20]{parreau:articolo}). Moreover, since the Bass--Serre tree $\Delta^{BS}(SL(2,\C_\omega))$ associated to $SL(2,\C_\omega)$ is naturally isometric to $C_\omega(\mathbb{H}^3,d/\lambda_k,O)$, as shown in~\cite[Proposition 3.21]{parreau:articolo}, the space $\mathbb{P}^1(\C)_\omega$ can be indentified also with $\mathbb{P}^1(\C_\omega)$ and this identification is compatible with the actions of $SL(2,\C)_\omega$ and $SL(2,\C_\omega)$, respectively. 

\subsection{Bounded cohomology of locally compact groups}\label{measurable}

From now until the end of this section we denote by $G$  a locally compact group. We endow $\R$ with the structure of a trivial normed $G$-module, where the considered norm is the standard Euclidean one. The space of bounded continuous functions is

\[
C_{cb}^n(G,\R):=C_{cb}(G^{n+1},\R)=\{ f: G^{n+1} \rightarrow \R | \text{$f$ is continuous and} \hspace{5pt} ||f||_\infty<\infty \}
\]
where the supremum norm is defined as

\[
||f||_\infty:=\sup_{g_0,\ldots,g_n \in G}|f(g_0,\ldots,g_n)|
\]
and $C^n_{cb}(G,\R)$ is endowed with the following $G$-module structure
\[
(g.f)(g_0,\ldots,g_n):=f(g^{-1}g_0,\ldots,g^{-1}g_n)
\]
for every element $g \in G$ and every function $f \in C^{n}_{cb}(G,\R)$ (here the notation $g.f$ stands for the action of the element $g$ on $f$). We denote by $\delta_n$ the homogeneous boundary operator of degree $n$, namely

\[
\delta_n:C^{n}_{cb}(G,\R) \rightarrow C^{n+1}_{cb}(G,\R),
\]
\[
\delta_n f(g_0,\ldots,g_{n+1})=\sum_{i=0}^{n+1} (-1)^i f(g_0,\ldots,\hat g_i, \ldots g_{n+1}),
\]
where the notation $\hat g_i$ indicates that the element $g_i$ has been omitted. 

There is a natural embedding of $\mathbb{R}$ into $C^0_{cb}(G,\R)$ given by the constant functions on $G$. This allows us to consider the following chain complex of $G$-modules

\[
\xymatrix{
0  \ar[r] & \R \ar[r] &  C^{0}_{cb}(G,\R) \ar[r]^{\delta_0} & C^{1}_{cb}(G,\R) \ar[r]^{\hspace{15pt} \delta_1} & \ldots
}
\]
and thanks to the compatibility of $\delta_n$ with respect to the $G$-action, we can consider the submodules of $G$-invariant vectors 

\[
\xymatrix{
0  \ar[r] &  C^{0}_{cb}(G,\R)^G \ar[r]^{\delta_0} & C^{1}_{cb}(G,\R)^G \ar[r]^{\delta_1} & C_{cb}^2(G,\R)^G \ar[r]^{\hspace{15pt} \delta_2} & \ldots
}
\]

Like in any other chain complex, we define the set of the $n^{th}$-\textit{bounded continuous cocycles} as

\[
Z^n_{cb}(G,\R)^G:=\text{ker}\Big(\delta_n:C^{n}_{cb}(G,\R)^G \rightarrow C^{n+1}_{cb}(G,\R)^G\Big)
\]
and the set of the $n^{th}$-\textit{bounded continuous coboundaries} 

\[
B^n_{cb}(G,\R)^G:=\text{im}\Big(\delta_{n-1}:C^{n-1}_{cb}(G,\R)^G \rightarrow C^{n}_{cb}(G,\R)^G \Big), \hspace{5pt} \text{and} \hspace{5pt} B^0_{cb}(G,\R):=0.
\]

\begin{deft}
The \textit{continuous bounded cohomology} in degree $n$ of $G$ with real coefficients is the space

\[
H^n_{cb}(G):=H^n_{cb}(G,\R)=\frac{Z^n_{cb}(G,\R)^G}{B^n_{cb}(G,\R)^G},
\]
with the quotient seminorm 

\[
||[f]||_\infty:=\inf ||f||_\infty,
\]
where the infimum is taken over all the possible representatives of $[f]$. 
\end{deft}

It is possible to gain information about the bounded cohomology of $G$ also by studying suitable spaces on which $G$ acts. More precisely, let $X$ be a measurable space on which $G$ acts measurably, that is the action map $\theta: G \times X \rightarrow X$ is measurable ($G$ is equipped with the $\sigma$-algebra of the Haar measurable sets). We set 

\[
\mathcal{B}^\infty(X^n,\R):=\{ f:X^n \rightarrow \R| f \hspace{3pt} \text{is measurable and} \hspace{3pt} \sup_{x \in X^n} |f(x)|< \infty \},
\]
 and we endow it with the structure of Banach $G$-module given by

\[
(g.f)(x_1,\ldots,x_n):=f(g^{-1}.x_1,\ldots,g^{-1}.x_n),
\]
for every $g \in G$ and every $f \in \mathcal{B}^\infty(X^n,\R)$. If \mbox{$\delta_n:\mathcal{B}^\infty(X^n,\R) \rightarrow \mathcal{B}^\infty(X^{n+1},\R)$} is the standard homogeneous coboundary operator, for $n\geq 1$ and $\delta_0:\R \rightarrow \mathcal{B}^\infty(X,\R)$  is the inclusion given by constant functions, we get a cochain complex $(\mathcal{B}^\infty(X^\bullet,\R),\delta_\bullet)$. We denote by $\mathcal{B}^\infty_\textup{alt}(X^{n+1},\R)$ the Banach $G$-submodule of alternating cochains, that is  the set of elements satisfying

\[
f(x_{\sigma(0)},\ldots,x_{\sigma(n)})=\textup{sgn}(\sigma)f(x_0,\ldots,x_n),
\]
for every permutation $\sigma \in S_{n+1}$.

\begin{deft}
Let $E$ be a Banach $G$-module. The \textit{continuous submodule} of $E$ is defined by

\[
\mathcal{C}E:=\{ v \in E| \lim_{g \to e} ||g.v -v||=0\}.
\]

A \textit{resolution} of $E$ is an exact complex $(E^\bullet, \partial_\bullet)$ of Banach $G$-modules such that $E^0=E$ and $E^n=0$ for every $n \leq -1$. 

\[
\xymatrix{
0 \ar[r] & E \ar[r]^{\partial_0} & E^1 \ar[r]^{\partial_1} & E^2 \ar[r]^{\partial_2} & \ldots \\
}
\]

We say that $(E^\bullet,\partial_\bullet)$ is a \textit{strong resolution} if the continuous subcomplex $(\mathcal{C}E^\bullet,\partial_\bullet)$ admits a contracting homotopy, that is a sequence of maps $h_n:\mathcal{C}E^{n+1} \rightarrow \mathcal{C}E^n$ such that $||h_n|| \leq 1 $ and $h_{n+1} \circ \partial_n + \partial_n \circ h_{n-1}=\textup{id}_{E^n}$ for all $n \in \N$.
\end{deft}

\indent In~\cite[Proposition 2.1]{burger:articolo} the authors prove that the complex of bounded measurable functions $(\mathcal{B}^\infty(X^\bullet,\R),\delta_\bullet)$ is a strong resolution of $\R$. Since the homology of any strong resolution of the trivial Banach $G$-module $\R$ maps in a natural way to the continuous bounded cohomology of $G$ by~\cite[Proposition 1.5.2.]{burger2:articolo}, there exists a canonical map 

\[
\mathfrak{c}^\bullet:H^\bullet(B^\infty(X^{\bullet+1},\R)^G) \rightarrow H^\bullet_{cb}(G).
\]

More precisely, every bounded measurable $G$-invariant cocycle $f:X^{n+1} \rightarrow \R$ determines canonically a class $\mathfrak{c}^n[f] \in H^n_{cb}(G)$. The same result holds for the subcomplex $(\mathcal{B}^\infty_\textup{alt}(X^\bullet,\R),\delta_\bullet)$ of alternating cochains. 


\section{The $\omega$-Borel cocycle}

\subsection{The cocycle $\text{Vol}^\omega$}

From now until the end of the paper we will consider the spaces $\mathbb{P}^1(\C)_\omega$ and $\mathbb{P}^1(\C_\omega)$ identified, hence we will refer to any of these two as they were the same space. The same will be done also for the groups $SL(n,\C)_\omega$ and $SL(n,\C_\omega)$. Moreover, to avoid a heavy notation we are going to refer to any sequence $(x_l)_{l \in \N}$ by dropping the parenthesis every time that we are considering the sequence itself instead of any of its single term.\\
\indent In this section we are going to construct a generalization of the hyperbolic volume function which will live on $\mathbb{P}^1(\C_\omega)^4$. This generalization will reveal the fundamental tool to define the $\omega$-Borel cocycle.\\ 
\indent  Before starting, we want to underline a delicate point. Since we want to exploit the properties of the standard Borel cocycle, one could try to define the new function $\textup{Vol}^\omega$ simply by taking the $\omega$-limit of the volumes, that is $\textup{Vol}^\omega(x^0_\omega,\ldots,x^3_\omega)=\text{$\omega$-$\lim_{l \to \infty}$}\textup{Vol}(x^0_l,\ldots,x^3_l)$, where $x^i_l$ is any representative of $x^i_\omega$. Unfortunately this definition is not correct. Indeed, if we suppose to have $3$ points that coincide, say $x^0_\omega=x^1_\omega=x^2_\omega$, different choices of representatives lead to different values of the $\omega$-limit of their volumes. Hence, we need to be careful. \\
\indent Let $\mathbb{P}^1(\C_\omega)^{(4)}$ be the space of $4$-tuples of distinct points on $\mathbb{P}^1(\C_\omega)$. As in the standard case, there is a natural cross ratio function 

\[
cr_\omega: \mathbb{P}^1(\C_\omega)^{(4)} \rightarrow \C_\omega \setminus \{0,1\}, \hspace{10pt} cr_\omega(x^0_\omega,x^1_\omega,x^2_\omega,x^3_\omega)=\frac{(x^0_\omega-x^2_\omega)(x^1_\omega-x^3_\omega)}{(x^0_\omega-x^3_\omega)(x^1_\omega-x^2_\omega)},
\]
which is well defined by its purely algebraic nature. Every $x^i_\omega$ may be considered in $\C_\omega$ or equal to $\infty$.  If we define the Bloch--Wigner function by

\[
D_2:\C \rightarrow \R, \hspace{10pt} D_2(z):=\Im(\textup{Li}_2(z))+\textup{arg}(1-z)\log |z|,
\]
where $\textup{Li}_2(z)$ is the dilogarithm function, by still denoting $D_2$ its continuous extension on $\mathbb{P}^1(\C)$, we can formulate the following  

\begin{deft}
The $\omega$-\textit{Bloch--Wigner function} is given by

\[
D_2^\omega:\C_\omega \cup \{ \infty \} \rightarrow \R, \hspace{10pt} D_2^\omega(x_\omega):=\text{$\omega$-$\lim_{l \to \infty}$} D_2(x_l) \hspace{5pt} \text{for $x_\omega \in \C_\omega$ and} \hspace{5pt} D^\omega_2(\infty):=0.
\]
where $x_l$ is any representative of the equivalence class $x_\omega$. 
\end{deft}

\begin{lem}
If $x_l$ and $y_l$ are two sequences representing the same element in $\C_\omega$, then 

\[
\text{$\omega$-$\lim_{l \to \infty}$} D_2(x_l)=\text{$\omega$-$\lim_{l \to \infty}$} D_2(y_l).
\]
\end{lem}

\begin{proof}
Since $\mathbb{P}^1(\C)$ is compact and $\omega$-$\lim_{l \to \infty} |x_l - y_l|^\frac{1}{\lambda_l}=0$, both sequences $x_l$ and $y_l$ will converge to the same limit in $\C \cup \{\infty \}$. Denote by $\xi$ this point. As a consequence of Proposition~\ref{continuous} and by the continuity of $D_2$ we have

\[
\text{$\omega$-$\lim_{l \to \infty}$} D_2(x_l)= D_2(\text{$\omega$-$\lim_{l \to \infty}$} x_l)=D_2(\xi)=D_2(\text{$\omega$-$\lim_{l \to \infty}$} y_l)=\text{$\omega$-$\lim_{l \to \infty}$} D_2(y_l), 
\]
as claimed.
\end{proof} 

The previous lemma guarantees that the definition of the $\omega$-Bloch--Wigner function is correct since it does not depend on the choice of the representative of the class $x_\omega$.

\begin{deft}
The $\omega$-\textit{volume function} for a $4$-tuple of points $(x^0_\omega,x^1_\omega,x^2_\omega,x^3_\omega) \in \mathbb{P}^1(\C_\omega)^4$ is defined as 

\[
 \textup{Vol}^\omega(x^0_\omega,x^1_\omega,x^2_\omega,x^3_\omega)=
\begin{cases}
&D^\omega_2(cr_\omega(x^0_\omega,x^1_\omega,x^2_\omega,x^3_\omega)) \hspace{5pt} \textup{if $(x^0_\omega,x^1_\omega,x^2_\omega,x^3_\omega) \in \mathbb{P}^1(\C_\omega)^{(4)}$},\\
&0 \hspace{5pt} \textup{otherwise}.
\end{cases}
\] 
\end{deft}

\begin{oss}\label{equality}
We are going to denote by $\textup{Vol}$ the composition $D_2 \circ cr$, where $D_2$ is the standard Bloch--Wigner function and $cr$ is the cross ratio on $\mathbb{P}^1(\C)$. 
Fix a $4$-tuple $(x^0_\omega,\ldots,x^3_\omega) \in \mathbb{P}^1(\C_\omega)^{4}$ of distinct points. Thanks to the natural identification between $\mathbb{P}^1(\C_\omega)$ and $\mathbb{P}^1(\C)_\omega$, we can think of each $x^i_\omega$ as the class of a sequence $x^i_l$ of points in $\mathbb{P}^1(\C)$. Now, it easy to see that

\[
cr_\omega(x^0_\omega,\ldots,x^3_\omega)=[cr(x^0_l,\ldots,x^3_l)]
\]
in $\C_\omega$ (if the $x^i_\omega$ are all distinct, also the terms of the sequences $x^i_l$ are distinct $\omega$-almost every $l \in \N$). By exploiting the previous identity, we can rewrite the definition of $\textup{Vol}^\omega$ as follows

\begin{align*}
\textup{Vol}^\omega(x^0_\omega,\ldots,x^3_\omega) &=D^\omega_2(cr_\omega(x^0_\omega,\ldots,x^3_\omega))=\text{$\omega$-$\lim_{l \to \infty}$} D_2(cr(x^0_l,\ldots,x^3_l))\\
&=\text{$\omega$-$\lim_{l \to \infty}$} \textup{Vol}(x^0_l,\ldots,x^3_l),
\end{align*}
and this is completely independent of the choice of representatives $x^0_l,\ldots,x^3_l$. Hence $\textup{Vol}^\omega$ coincides with the $\omega$-limit of the standard volumes $\textup{Vol}(x^0_l,\ldots,x^3_l)$ on a $4$-tuple $(x^0_\omega,\ldots,x^3_\omega) \in \mathbb{P}^1(\C_\omega)^{(4)}$, where $x^i_l$ is any representative for $x^i_\omega$. Even though we have already underlined that this is not true on the whole space $\mathbb{P}^1(\C_\omega)^4$, we can always choose suitable representatives for $x^i_\omega$ such that

\[
\textup{Vol}^\omega(x^0_\omega,\ldots,x^3_\omega)=\text{$\omega$-$\lim_{l \to \infty}$} \textup{Vol}(x^0_l,\ldots,x^3_l).
\] 
\end{oss}

\begin{prop}
The function $\textup{Vol}^\omega$ is a bounded, alternating, $GL(2,\C_\omega)$-invariant cocycle. 
\end{prop}

\begin{proof}
Most of the properties we stated follow directly from the properties of the standard volume function $\textup{Vol}$. We are going to show $GL(2,\C_\omega)$-invariance, for instance. From now until the end of the proof we are going to pick suitable representative sequences for points in $\mathbb{P}^1(\C_\omega)$ such that

\[
\textup{Vol}^\omega(x^0_\omega,\ldots,x^3_\omega)=\text{$\omega$-$\lim_{l \to \infty}$} \textup{Vol}(x^0_l,\ldots,x^3_l).
\] 

Let $g_\omega \in GL(2,\C_\omega)$. We want to show that $g_\omega . \textup{Vol}^\omega = \textup{Vol}^\omega$. 

\begin{align*}
g_\omega.\textup{Vol}^\omega(x^0_\omega,x^1_\omega,x^2_\omega,x^3_\omega)&=\textup{Vol}^\omega(g_\omega^{-1}.x^0_\omega,\ldots,g_\omega^{-1}.x^3_\omega)\\
&=\text{$\omega$-$\lim_{l \to \infty}$} \textup{Vol}(g_l^{-1}.x^0_l,\ldots,g_l^{-1}.x^3_l)
\end{align*}
and thanks to the equivariance of the classic volume function we get 

\[
\text{$\omega$-$\lim_{l \to \infty}$} \textup{Vol}(g_l^{-1}.x^0_l,\ldots,g_l^{-1}.x^3_l)=\text{$\omega$-$\lim_{l \to \infty}$} \textup{Vol}(x^0_l,\ldots,x^3_l)=\textup{Vol}^\omega(x^0_\omega,\ldots,x^3_\omega),
\]
as required. The strategy to prove the alternating property and the cocycle property of $\textup{Vol}^\omega$ is the same as above and we omit it.\\
\indent Finally, the boundedness is obvious since the $\omega$-Bloch--Wigner is nothing more than the $\omega$-limit of a sequence of real values all bounded by $\nu_3$ on $\mathbb{P}^1(\C_\omega)^{(4)}$ and it coincides with $0$ on the complementary. Here $\nu_3$ is the volume of a regular ideal hyperbolic tetrahedron in $\mathbb{H}^3$.

\end{proof}

\subsection{The cocycle $B^\omega_n$}

In order to define the $\omega$-Borel invariant for a representation $\rho_\omega:\Gamma \rightarrow SL(n,\C_\omega)$, we first need to define the $\omega$-Borel cocycle. We are going to follow the same construction exposed in~\cite[Section 3]{iozzi:articolo}. Let $\mathfrak{S}^\omega_k(m)$ be the following space

\[
\mathfrak{S}^\omega_k(m):=\{ (x^0_\omega,\ldots,x_\omega^k) \in (\C_\omega^m)^{k+1} | \langle x^0_\omega,\ldots x_\omega^k \rangle = \C_\omega^m \} / GL(m,\C_\omega)
\]
where $GL(m,\C_\omega)$ acts on $(k+1)$-tuples of vectors by the diagonal action and $\langle x^0_\omega,\ldots x^k_\omega \rangle$ is the $\C_\omega$-linear space generated by $x^0_\omega,\ldots,x^k_\omega$. It obvious that if $k < m-1$ the space defined above is empty. For every $m$-dimensional vector space $V$ over $\C_\omega$ and any $(k+1)$-tuple of spanning vectors $(x^0_\omega,\ldots,x^k_\omega) \in V^{k+1}$, we choose an isomorphism $V \rightarrow \C^m_\omega$. Since any two different choices of isomorphisms are related by an element $g_\omega \in GL(m,\C_\omega)$, we get a well defined element of $\mathfrak{S}^\omega_k(m)$ which will be denoted by $[V;(x^0_\omega,\ldots,x^k_\omega)]$. For 

\[
\mathfrak{S}^\omega_k:= \bigsqcup_{m \geq 0} \mathfrak{S}^\omega_k(m)= \mathfrak{S}^\omega_k(0) \sqcup \ldots \sqcup \mathfrak{S}^\omega_k(k+1)
\]
we have two different face maps $\varepsilon^{(k)}_i, \eta^{(k)}_i: \mathfrak{S}^\omega_k \rightarrow \mathfrak{S}^\omega_{k-1}$ given by

\begin{align*}
\varepsilon^{(k)}_i[\C^m_\omega;(x^0_\omega,\ldots,x^k_\omega)]&:=[\langle x^0_\omega, \ldots, \hat x^i_\omega, \ldots, x^k_\omega \rangle;(x^0_\omega,\ldots, \hat x^i_\omega, \ldots, x^k_\omega)],\\
\eta^{(k)}_i[\C^m_\omega;(x^0_\omega,\ldots,x^k_\omega)]&:=[\C^m_\omega/\langle x^i_\omega \rangle;(x^0_\omega,\ldots, \hat x^i_\omega, \ldots, x^k_\omega)].
\end{align*}

Since these maps satisfy the same relations as in~\cite{iozzi:articolo}, that is for all $0 \leq i< j \leq k$ 

\begin{align*}
\varepsilon^{(k-1)}_j \varepsilon^{(k)}_i &= \varepsilon^{(k-1)}_i \varepsilon^{(k)}_{j+1},\\
\eta^{(k-1)}_j \eta^{(k)}_i &= \eta^{(k-1)}_i \eta^{(k)}_{j+1},\\
\eta^{(k-1)}_j \varepsilon^{(k)}_i &= \varepsilon^{(k-1)}_i \eta^{(k)}_{j+1},
\end{align*}
we can define a boundary operator 

\[
D_k: \Z[\mathfrak{S}^\omega_k] \rightarrow \Z[\mathfrak{S}^\omega_{k-1}], \hspace{10pt} D_k(\sigma):=\sum_{i=0}^k (-1)^i (\varepsilon^{(k)}_i(\sigma) - \eta^{(k)}_i(\sigma)),
\]
where $\Z[\mathfrak{S}^\omega_k]$ is the free abelian group generated by $\mathfrak{S}^\omega_k$ and it is equal to $0$ for $k \leq -1$. We still denote by $\varepsilon^{(k)}_i$ and $\eta^{(k)}_i$ the linear extensions of face maps to $\Z[\mathfrak{S}^\omega_k]$. In this way we have constructed a chain complex $(\Z[\mathfrak{S}^\omega_\bullet],D_\bullet)$. With the purpose of dualizing this complex, we recall that we have a natural action of the symmetric group $S_{k+1}$ on $\mathfrak{S}^\omega_k$, hence we can define  

\[
\R_\text{alt}(\mathfrak{S}^\omega_k):= \{ f: \mathfrak{S}^\omega_k \rightarrow \R | f \hspace{5pt} \text{is alternating with respect to the $S_{k+1}$-action} \} 
\]
and we can define $D^*_k$ as the dual of $D_k \otimes id_\R$. The construction above produces a cochain complex $(\R_\text{alt}(\mathfrak{S}^\omega_\bullet),D^*_\bullet)$.\\
\indent We are going now to define a cocycle living in $\R_\text{alt}(\mathfrak{S}^\omega_3)$ which will be used to construct the $\omega$-Borel cocycle.  Since the $\omega$-volume function $\textup{Vol}^\omega$ introduced in the previous section can be thought of as defined on $(\C_\omega^2 \setminus \{0\})^4$, it is extendable to 

\[
\text{Vol}^\omega: \mathfrak{S}^\omega_3 \rightarrow \R
\]
where we set $\text{Vol}^\omega|\mathfrak{S}^\omega_3(m)$ to be identically zero if $m \neq 2$ and 

\[
\text{Vol}^\omega[\C^2_\omega;(v^0_\omega,\ldots,v^3_\omega)]:=
	\begin{cases}
	&\text{Vol}^\omega(v^0_\omega,\ldots,v^3_\omega) \hspace{10pt} \text{if each $v^i_\omega \neq 0$,}\\
	&0 \hspace{10pt} \text{otherwise}.
	\end{cases}
\]

By the compatibilty of the $\omega$-limit with respect to finite sums, it should be clear that 

\begin{prop}
The function $\textup{Vol}^\omega \in \R_\textup{alt}(\mathfrak{S}^\omega_3)$ is a cocycle, that is $D^*_4(\textup{Vol}^\omega)=0$.
\end{prop}

Since the proof of this proposition is the same as~\cite[Lemma 8, Lemma 9]{iozzi:articolo} we omit it. In order to define the $\omega$-Borel cocyle we are going to introduce the spaces of affine flags in $\C^n_\omega$. A complete flag $F_\omega$ in $\C^n_\omega$ is a sequence of linear subspaces

\[
F^0_\omega \subset F^1_\omega \subset \ldots \subset F^n_\omega
\]
such that every $F^i_\omega$ has dimension $i$ as $\C_\omega$-vector space. An affine flag $(F_\omega,v_\omega)$ is a complete flag $F_\omega$ together with an $n$-tuple of vectors $v_\omega=(v^1_\omega,\ldots,v^n_\omega) \in (\C_\omega^n)^n$ such that 

\[
F^i_\omega=\C_\omega v_\omega^i + F_\omega^{i-1}, \hspace{10pt} i\geq1.
\]

It is clear that the group $GL(n,\C_\omega)$ acts naturally on the space of flags $\mathscr{F}(n,\C_\omega)$ and on the space of affine flags $\mathscr{F}_\text{aff}(n,\C_\omega)$ of $\C_\omega^n$. Let $\Z[\mathscr{F}_\text{aff}(n,\C_\omega)^{k+1}]$ be the abelian group generated by $\mathscr{F}_\text{aff}(n,\C_\omega)^{k+1}$ and let $\partial_k$ be the standard boundary map induced by the face maps $\varepsilon^{(k)}_i:\mathscr{F}_\text{aff}(n,\C_\omega)^{k+1} \rightarrow \mathscr{F}_\text{aff}(n,\C_\omega)^k$ consisting in dropping the $i^{th}$-component for $1 \leq k \leq n-1$. Moreover set $\partial_0:\Z[\mathscr{F}_\text{aff}(n,\C_\omega)] \rightarrow 0$. We are ready now to define

\[
T_k:(\Z[\mathscr{F}_\text{aff}(n,\C_\omega)^k],\partial_k) \rightarrow (\Z[\mathfrak{S}^\omega_k],D_k)
\]
which will enable us to construct a morphism between the dual of the complexes above (more precisely on their alternating versions). Given a multi-index $\mathbf{J} \in \{0,1,\ldots,n-1\}^{k+1}$, we start by defining 
\[
\tau_\mathbf{J}: \mathscr{F}_\text{aff}(n,\C_\omega)^{k+1} \rightarrow \mathfrak{S}^\omega_k
\]
as the function
\[
\tau_\mathbf{J}((F_{0,\omega},v_{0,\omega}),\ldots,(F_{k,\omega},v_{k,\omega})):=
\Bigg[
\frac{\langle F^{j_0+1}_{0,\omega},\ldots,F^{j_k+1}_{k,\omega} \rangle}{\langle F^{j_0}_{0,\omega}, \ldots, F^{j_k}_{k,\omega} \rangle}
; (v^{j_0+1}_{0,\omega},\ldots,v^{j_k+1}_{k,\omega})
\Bigg]
\]
and finally
\[
T_k((F_{0,\omega},v_{0,\omega}),\ldots,(F_{k,\omega},v_{k,\omega})):=\sum_{\mathbf{J} \in \{0, \ldots, n-1 \}^{k+1}} \tau_\mathbf{J}((F_{0,\omega},v_{0,\omega}),\ldots,(F_{k,\omega},v_{k,\omega})).
\]

If we now recall that there exists a natural action of $S_{k+1}$ on $\mathscr{F}_\text{aff}(n,\C_\omega)^{k+1}$ and dualize the complex considered so far, we get the cocomplex of alternating cochains $(\R_\text{alt}(\mathscr{F}_\text{aff}(n,\C_\omega)^{k+1}), \partial_k^*)$ (here $\partial_k^*$ is the dual of $\partial_k \otimes id_\R$). By denoting $T^*_k$ the dual map of $T_k \otimes id_\R$, the same proof of~\cite[Lemma 11]{iozzi:articolo} guarantees that $T^*_k$ is a morphism a complexes taking values in $(\R_\text{alt}(\mathscr{F}_\text{aff}(n,\C_\omega)^{k+1}))^{GL(n,\C_\omega)}$. 

\begin{deft}
We define the $\omega$-\textit{Borel function} of degree $n$ as 

\begin{align*}
& B_n^\omega ((F_{0,\omega},v_{0,\omega}),\ldots,(F_{3,\omega},v_{3,\omega})):=T^*_3(\textup{Vol}^\omega)=\\
&=\sum_{\mathbf{J} \in \{0, \ldots ,n-1\}^4} 
\textup{Vol}^\omega
\Bigg[
\frac{\langle F^{j_0+1}_{0,\omega},\ldots,F^{j_3+1}_{3,\omega} \rangle}{\langle F^{j_0}_{0,\omega}, \ldots, F^{j_3}_{3,\omega} \rangle}
; (v^{j_0+1}_{0,\omega},\ldots,v^{j_3+1}_{3,\omega})
\Bigg].
\end{align*} 
\end{deft}

Using the same approach of~\cite{iozzi:articolo} it is straghtfoward to prove that
\begin{prop}
The function $B^\omega_n$ is a bounded, alternating, strict $GL(n,\C_\omega)$-invariant cocycle on the space $\mathscr{F}_{\textup{aff}}(n,\C_\omega)^4$ of $4$-tuples of affine flags which naturally descends to the space $\mathscr{F}(n,\C_\omega)^4$ of $4$-tuples of flags. Moreover, for every $4$-tuple of flags $(F_{0,\omega},\ldots,F_{3,\omega}) \in \mathscr{F}(n,\C_\omega)^4$ we have the following bound

\[
|B^\omega_n(F_{0,\omega},\ldots,F_{3,\omega})|\leq \frac{n(n^2-1)}{6}\nu_3.
\]
\end{prop}

We want now to use~\cite[Proposition 2.1]{burger:articolo} in order to obtain the desired cohomology class. Before doing this we need to underline a delicate point in the discussion. By Proposition~\ref{locally} the field $\C_\omega$ is not locally compact with respect to the topology induced by the ultrametric absolute value. In particular the group $SL(n,\C_\omega)$ cannot be locally compact with respect to the topology inherited by $M(n,\C_\omega)$ seen as $\C_\omega^{n^2}$. Hence it is meaningless to refer to the Haar measure or to the Haar $\sigma$-algebra for $SL(n,\C_\omega)$. In order to overcome these difficulties, we are going to consider $SL^\delta(n,\C_\omega)$, that is the group $SL(n,\C_\omega)$ endowed with the discrete topology. The same for $GL^\delta(n,\C_\omega)$. Moreover, in order to apply correctly~\cite[Proposition 2.1]{burger:articolo}, we are going to consider the discrete $\sigma$-algebra on both $\mathfrak{S}^\omega_k$ and $\mathscr{F}(n,\C_\omega)$.\\
\indent Recall that $\mathfrak{S}^\omega_k(n)$ is a space on which the symmetric group $S_{k+1}$ acts naturally. Let $\mathcal{B}_\textup{alt}^\infty(\mathfrak{S}^\omega_k)$ be the Banach space of bounded alternating Borel functions on $\mathfrak{S}^\omega_k$. The restriction of $D^*_k$ gives us back a complex of Banach spaces $(\mathcal{B}^\infty_\textup{alt}(\mathfrak{S}^\omega_\bullet),D^*_\bullet)$.\\
By restricting the map $T^*_k$ to the subcomplexes of bounded Borel functions and by applying~\cite[Proposition 2.1]{burger:articolo} to $(\mathcal{B}^\infty_\textup{alt}(\mathscr{F}(n,\C_\omega)^{\bullet+1}),\partial_\bullet)$, we get a map

\[
S^k_\omega(n):H^k(\mathcal{B}^\infty_\textup{alt}(\mathfrak{S}^\omega_\bullet)) \rightarrow H^k_b(GL^\delta(n,\C_\omega)).
\] 

\begin{deft}
With the notation above, we define the $\omega$-\textit{Borel cohomology class of degree} $n$ as

\[
\beta ^\omega(n):=S^3_\omega(n)(\textup{Vol}^\omega)=\mathfrak{c}^3[B^\omega_n],
\]
where $\mathfrak{c^3}:H^3(\mathcal{B}^\infty_\textup{alt}(\mathscr{F}(n,\C_\omega)^{\bullet+1})^{GL(n,\C_\omega)}) \rightarrow H^3_b(GL^\delta(n,\C_\omega))$ is the canonical map of~\cite[Proposition 2.1]{burger:articolo}. 
\end{deft}

\begin{oss}

We have the following commutative diagram 

\[
\xymatrix{
1 \ar[r] & \C^\times_\omega \ar[r] & GL(n,\C_\omega) \ar[r] & PGL(n,\C_\omega) \ar[r] \ar[d]^{\cong} & 1\\
1 \ar[r] & \mu_n \ar[u] \ar[r] & SL(n,\C_\omega) \ar[u] \ar[r] & PSL(n,\C_\omega) \ar[r] & 1\\
}
\]
where $\C_\omega^\times$ is the group of invertible elements of $\C_\omega$ and $\mu_n$ is the group of the $n$-th roots of unity. Since these groups are both amenable, by functoriality of bounded cohomology it is possible to conclude that $H_b^3(GL^\delta(n,\C_\omega)) \cong H^3_b(SL^\delta(n,\C_\omega))$. In particular, we are going to think of the class $\beta^\omega(n)$ as an element of both $H^3_b(GL^\delta(n,\C_\omega))$ and $H^3_b(SL^\delta(n,\C_\omega))$. 
\end{oss}


\section{The $\omega$-Borel invariant for a representation $\rho_\omega$}

Let $\Gamma$ be the fundamental group of a complete hyperbolic 3-manifold $M$ with toric cusps. This means that we can decompose the manifold $M$ as  $M = N \cup \bigcup_{i=1}^h C_i$, where $N$ is any compact core of $M$ and for every $i=1,\ldots,h$ the component $C_i$ is a cuspidal neighborhood diffeomorphic to $T_i \times (0,\infty)$, where $T_i$ is a torus whose fundamental group corresponds to a suitable abelian parabolic subgroup of $PSL(2,\C)$. Our aim is to define a numerical invariant associated to any representation $\rho_\omega: \Gamma \rightarrow SL(n,\C_\omega)$. Let $i: (M, \varnothing) \rightarrow (M,M \setminus N)$ be the natural inclusion map. Since the fundamental group of the boundary $\partial N$ is abelian, hence amenable, it can be proved that the maps $i^*_b:H^k_b(M,M \setminus N) \rightarrow H^k_b(M)$ induced at the level of bounded cohomology groups are isometric isomorphisms for $k \geq 2$ (see~\cite{bucher:articolo}). Moreover, it holds $H^k_b(M,M \setminus N) \cong H^k_b(N, \partial N)$ by homotopy invariance of bounded cohomology. If we denote by $c$ the canonical comparison map $c:H^k_b(N,\partial N) \rightarrow H^k(N,\partial N)$, we can consider the composition

\[
\xymatrix{
H^3_{b}(SL^\delta(n,\C_\omega)) \ar[r]^{(\rho_\omega)^*_b} & H^3_b(\Gamma) \cong H^3_b(M) \ar[r]^{\hspace{15pt}(i^*_b)^{-1}} & H^3_b(N,\partial N) \ar[r]^c & H^3(N,\partial N),
}
\]
where the isomorphism that appears in this composition holds since $M$ is aspherical. By choosing a fundamental class $[N,\partial N]$ for $H_3(N,\partial N)$ we are ready to give the following 

\begin{deft}
The $\omega$-\textit{Borel invariant} associated to a representation \mbox{$\rho_\omega: \Gamma \rightarrow SL(n,\C_\omega)$} is given by 

\[
\beta^\omega_n(\rho_\omega):=\langle (c \circ (i^*_b)^{-1} \circ (\rho_\omega)^*_b)\beta^\omega(n),[N,\partial N] \rangle,
\]
where the brackets $\langle \cdot , \cdot \rangle$ indicate the Kronecker pairing. 
\end{deft}

\begin{oss}
The previous definition is indipendent of the choice of the compact core $N$. Moreover, it can be easily extended to any lattice of $PSL(2,\C)$. 
\end{oss}

We are going to generalize some of the classic results valid for the standard Borel invariant. The proofs are identical to the ones exposed in~\cite{iozzi:articolo}. Before starting, we recall the existence of natural transfer maps

\[
\xymatrix{
H^\bullet_b(\Gamma) \ar[r]^{\textup{trans}_\Gamma \hspace{25pt}} & H^\bullet_{cb}(PSL(2,\C)) \hspace{20pt} & H^\bullet(N,\partial N) \ar[r]^{\tau_{DR}\hspace{10pt}} & H^\bullet_c(PSL(2,\C)),
}
\]
where $H_c^\bullet(PSL(2,\C))$ denotes the continuous cohomology groups of $PSL(2,\C)$.  We remind the reader that the continuous cohomology groups of a locally compact group $G$ are constructed as the continuous bounded cohomology groups just by dropping the requirement of boundedness of cochains.\\
\indent The transfer maps are defined as it follows. Let $V_k$ be the set $C_b((\mathbb{H}^3)^{k+1},\R)$ of real bounded continuous functions on $(k+1)$-tuples of points of $\mathbb{H}^3$. With the standard homogeneous boundary operators and the structure of Banach $PSL(2,\C)$-module given by

\[
(g.f)(x^0,\ldots,x^n):=f(g^{-1}x^0,\ldots,g^{-1}x^n), \hspace{10pt} ||f||_{\infty}=\sup_{x^0,\dots,x^n \in \mathbb{H}^3} |f(x^0,\ldots,x^n)|
\]
for every $f \in C_b((\mathbb{H}^3)^{n+1},\R)$ and $g \in PSL(2,\C)$, we get a complex $V_\bullet=C_b((\mathbb{H}^3)^{\bullet+1},\R)$ of Banach $PSL(2,\C)$-modules that allows us to compute the continuous bounded cohomology of $PSL(2,\C)$. More precisely, it holds
\[
H^k(V_\bullet^{PSL(2,\C)}) \cong H^k_{cb}(PSL(2,\C))
\] 
for every $k \geq 0$. Moreover, by substituting $PSL(2,\C)$ with $\Gamma$, we have in an analogous way that

\[
H^k(V_\bullet^\Gamma) \cong H^k_{b}(\Gamma)
\] 
for every $k \geq 0$. The previous considerations allow us to define the map

\[
\textup{trans}_\Gamma:V_k^\Gamma \rightarrow V_k^{PSL(2,\C)},
\]
\[
\textup{trans}_\Gamma(c)(x_0,\ldots,x_n):=\int_{\Gamma \backslash PSL(2,\C)} c(\bar gx_0,\ldots,\bar gx_n)d\mu(\bar g),
\] 
where $c$ is any $\Gamma$-invariant element of $V_k$ and $\mu$ is any invariant probability measure on $\Gamma \backslash PSL(2,\C)$. Here $\bar g$ stands for the equivalence class of $g$ into $\Gamma \backslash PSL(2,\C)$.\\
\indent Since $\textup{trans}_\Gamma(c)$ is $PSL(2,\C)$-equivariant and $\textup{trans}_\Gamma$ commutes with the coboundary operator, we get a well-defined map

\[
\textup{trans}_\Gamma: H^\bullet_b(\Gamma) \rightarrow H^\bullet_{cb}(PSL(2,\C)).
\]

We now pass to the description of the map $\tau_{DR}$. If $\pi: \mathbb{H}^3 \rightarrow M =\Gamma \backslash \mathbb{H}^3$ is the natural covering projection, we set $U:=\pi^{-1}(M \setminus N)$. Recall that the relative cohomology group $H^k(N,\partial N)$ is isomorphic to the cohomology group $H^k(\Omega^\bullet(\mathbb{H}^3,U)^\Gamma)$ of the $\Gamma$-invariant differential forms on $\mathbb{H}^3$ which vanishes on $U$.  Since, by Van Est isomorphism we have that $H^k_c(PSL(2,\C),\R) \cong \Omega^k(\mathbb{H}^3)^{PSL(2,\C)}$, we define

\[
\tau_{DR}: \Omega^k(\mathbb{H}^3,U)^\Gamma \rightarrow \Omega^k(\mathbb{H}^3)^{PSL(2,\C)} , \hspace{10pt} \tau_{DR}(\alpha):=\int_{\Gamma \backslash PSL(2,\C)} \bar g ^*\alpha d\mu(\bar g),
\]
where $\mu$ and $\bar g$ are the same as before. The map $\tau_{DR}$ commutes with the coboundary operators inducing a map 

\[
\tau_{DR}:H^k(N,\partial N) \cong H^k(\Omega^\bullet(\mathbb{H}^3,U)^\Gamma) \rightarrow H^k(\Omega^\bullet(\mathbb{H}^3)^{PSL(2,\C)}) \cong H^k_c(PSL(2,\C)).
\]

For a more detailed description of the above maps we suggest to the reader to check~\cite[Section 3.2]{bucher2:articolo}.

\begin{prop}
For $k \geq 2$ the diagram

\[
\xymatrix{
H^k(\mathcal{B}^\infty_\textup{alt}(\mathfrak{S}^\omega_\bullet)) \ar[r]^{S^k_\omega(n+1) \hspace{15pt}} \ar[dr]_{S^k_\omega(n)} & H^k_b(GL^\delta(n+1,\C_\omega)) \ar[d]\\
&H^k_b(GL^\delta(n,\C_\omega))\\
}
\]
commutes. The vertical arrow is induced by the left corner injection $GL(n,\C_\omega) \rightarrow GL(n+1,\C_\omega)$. In particular we have that $\beta^\omega(n+1)$ restricts to $\beta^\omega(n)$. 
\end{prop}

\begin{proof}
Let $i_n: \C_\omega^n \rightarrow \C_\omega^{n+1}$ be the injection $i_n(x^1_\omega,\ldots,x^n_\omega):=(x^1_\omega,\ldots,x^n_\omega,0)$. By an abuse of notation we define 

\[
i_n:\mathscr{F}_\textup{aff}(n,\C_\omega) \rightarrow \mathscr{F}_\textup{aff}(n+1,\C_\omega)
\]
as $i_n((F_\omega,v_\omega))=(\tilde F_\omega,\tilde v_\omega)$ where for $0 \leq j \leq n$ we have $\tilde F_\omega^j=i_n(F^j_\omega),\tilde v_\omega^j=i_n(v^j_\omega)$ and $\tilde v^{n+1}_\omega=e_{n+1}$. 
If we set $\mathbf{J} \in \{0, \ldots n \}^{k+1}$ and $I=\{ i: 0 \leq i \leq k \hspace{5pt} \textup{such that} \hspace{5pt} j_i=n \}$, it is easy to verify that if $I=\varnothing$ this implies $\mathbf{J} \in \{0, \ldots ,n-1\}^{k+1}$ and 

\[
\tau_\mathbf{J}(i_n(F_{0,\omega},v_{0,\omega}),\ldots,i_n(F_{k,\omega},v_{k,\omega}))=\tau_\mathbf{J}((F_{0,\omega},v_{0,\omega}),\ldots,(F_{k,\omega},v_{k,\omega}))
\]
while if $I \neq \varnothing$, then

\[
\tau_\mathbf{J}(i_n(F_{0,\omega},v_{0,\omega}),\ldots,i_n(F_{k,\omega},v_{k,\omega}))=[\C_\omega;(\delta^I_0,\ldots,\delta^I_k)],
\]
where $\delta^I_i=[e_{n+1}]$ if $i \in I$ and $0$ otherwise. 
The previous considerations imply that $i_n$ induces a commutative diagram of complexes 

\[
\xymatrix{
\mathcal{B}^\infty_\textup{alt}(\mathfrak{S}^\omega_k) \ar[r]^{T^*_k \hspace{30pt}} \ar[dr]_{T^*_k} & \mathcal{B}^\infty_\textup{alt}(\mathscr{F}_\textup{aff}(n+1,\C_\omega)^{k+1}) \ar[d]^{i^*_n}\\
& \mathcal{B}^\infty_\textup{alt}(\mathscr{F}_\textup{aff}(n,\C_\omega)^{k+1})\\
}
\]
and since the map $i_n^*$ implements the restriction in bounded cohomology, the commutativity of the diagram which appears in the statement follows. In particular, by focusing our attention on the case of $k=3$ we get

\[
i^*_{n}(B^\omega_{n+1})=i^*_n \circ T^*_3 (\textup{Vol}^\omega)=T^*_3(\textup{Vol}^\omega)=B^\omega_{n}
\]
as claimed.
\end{proof}

\begin{prop}
For any representation $\rho_\omega:\Gamma \rightarrow SL(n,\C_\omega)$ the composition

\[
\xymatrix{
H^3_b(SL^\delta(n,\C_\omega)) \ar[r] & H^3_b(\Gamma) \ar[r]^{\textup{trans}_\Gamma \hspace{20pt}} & H^3_{cb}(PSL(2,\C))
}
\]
maps $\beta^\omega(n)$ to $\frac{\beta^\omega_n(\rho_\omega)}{\textup{Vol}(M)}\beta(2)$. In particular, it holds the following bound

\[
|\beta^\omega_n(\rho_\omega)| \leq \frac{n(n^2-1)}{6}\textup{Vol}(M),
\]
as in the classic case.
\end{prop}

\begin{proof}
Recall that we have the following commutative diagram

\[
\xymatrix{
H^3_b(SL^\delta(n,\C_\omega)) \ar[d]^{(\rho_\omega)^*_b}\\
H^3_b(\Gamma) \ar[d]^{\cong} \ar[dr]^{\text{trans}_\Gamma}\\
H^3_b(N,\partial N) \ar[d]^c & H^3_{cb}(PSL(2,\C)) \ar[d]^c\\
H^3(N,\partial N) \ar[r]^{\tau_{DR}} & H^3_{c}(PSL(2,\C)).\\
}
\]

Since $H^3_{cb}(PSL(2,\C)) \cong \R$, there exists a suitable $\lambda \in \R$ such that 
\[
\text{trans}_\Gamma \circ (\rho_\omega)^*_b (\beta^\omega(n))=\lambda \beta(2).
\]

Hence by composing both sides with the comparison map $c$, we obtain

\[
c \circ \text{trans}_\Gamma \circ (\rho_\omega)^*_b(\beta^\omega(n))=c(\lambda \beta(2))=\lambda(c\beta(2))=\lambda \beta(2).
\]

If we pick up $\omega_{N,\partial N} \in H^3(N,\partial N)$ in such a way that its evaluation on the fundamental class $[N,\partial N]$ gives us back $\textup{Vol}(M)$, we have that $\tau_{DR}(\omega_{N,\partial N})=\beta(2)$. In particular 

\[
\tau_{DR}(c \circ (i^*_b)^{-1} \circ (\rho_\omega)^*_b(\beta^\omega(n)))= \lambda \tau_{DR}(\omega_{N,\partial N})
\]
and by injectivity of the map $\tau_{DR}$ in top degree we get

\[
(c \circ (i^*_b)^{-1} \circ (\rho_\omega)^*_b)(\beta^\omega(n))=\lambda \omega_{N,\partial N}.
\]

If we evaluate both sides on the fundamental class, we obtain

\[
\beta^\omega_n(\rho_\omega)=\langle (c \circ (i^*_b)^{-1} \circ (\rho_\omega)^*_b)(\beta^\omega(n)),[N,\partial N] \rangle=\langle  \lambda \omega_{N,\partial N}, [N,\partial N] \rangle= \lambda \textup{Vol}(M). 
\]

At the same time it holds

\[
|\lambda|=\frac{|| \text{trans}_\Gamma \circ (\rho_\omega)^*_b \beta^\omega(n)||}{||\beta(2)||} \leq \frac{n(n^2-1)}{6},
\]
from which it follows

\[
|\beta^\omega_n(\rho_\omega)| \leq \frac{n(n^2-1)}{6}\textup{Vol}(M),
\]
as claimed. 
\end{proof}

Recall that there is a natural inclusion of fields of $\C$ into $\C_\omega$ given by constant sequences. In particular we have natural embeddings of $\C^m$ into $\C_\omega^m$ and of $SL(n,\C)$ into $SL(n,\C_\omega)$. Since every representation $\rho:\Gamma \rightarrow SL(n,\C)$ determines a representation $\hat \rho$ into $SL(n,\C_\omega)$ by composing it with the previous embedding, it is quite natural to ask which is the relation between $\beta_n^\omega(\hat \rho)$ and $\beta_n(\rho)$. We have the following 

\begin{prop}
Let $\rho:\Gamma \rightarrow SL(n,\C)$ be a representation. If we denote by $\hat \rho:\Gamma \rightarrow SL(n,\C_\omega)$ the representation obtained by composing $\rho$ with the natural embedding of $SL(n,\C)$ into $SL(n,\C_\omega)$, we have

\[
\beta^\omega_n(\hat \rho)=\beta_n(\rho).
\]
\end{prop}

\begin{proof}
We are going to prove that the cohomology class $\beta^\omega(n)$ restricts naturally to the class $\beta(n)$. Let $j:SL(n,\C) \rightarrow SL(n,\C_\omega)$ be the natural embedding. By endowing both spaces with the discrete topology, we have a continuous morphism of groups that induces a map 

\[
j^*_b:H^3_b(SL^\delta(n,\C_\omega)) \rightarrow H^3_b(SL^\delta(n,\C)).
\]
We want to prove that $j^*_b(\beta^\omega(n))=\beta(n)$. From this it will follow

\begin{align*}
\beta^\omega_n(\hat \rho)&=\langle (c \circ (i^*_b)^{-1} \circ \hat \rho^*_b)\beta^\omega(n),[N,\partial N] \rangle=\langle (c \circ (i^*_b)^{-1} \circ (j \circ \rho)^*_b)\beta^\omega(n),[N,\partial N] \rangle\\
&=\langle (c \circ (i^*_b)^{-1} \circ \rho^*_b  \circ j^*_b)\beta^\omega(n),[N,\partial N] \rangle= \langle (c \circ (i^*_b)^{-1} \circ \rho^*_b)\beta(n),[N,\partial N] \rangle=\beta_n(\rho).
\end{align*}

Similarly to what we have done for the field $\C_\omega$, we define the configuration space

\[
\mathfrak{S}_k(m):=\{ (x^0,\ldots,x^k) \in (\C^m)^{k+1}|\langle x^0,\ldots,x^k\rangle = \C^m\}/GL(m,\C).
\]
 for every $k \geq m-1$. This family of spaces is exactly the family introduced by~\cite{iozzi:articolo}. There exists a natural family of maps given by

\[
\hat j_k(m): \mathfrak{S}_k(m) \rightarrow \mathfrak{S}^\omega_k(m), \hspace{10pt} \hat j_k(m)[\C^m;(v^0,\ldots,v^k)]:=[\C^m_\omega;(v^0,\ldots,v^k)],
\]
where each vector $v^i$ which appears on the right-hand side of the equation is thought of as an element of $\C^m_\omega$. This function is well-defined because $v^0,\ldots,v^k$ are generators also for $\C^m_\omega$ as a $\C_\omega$-vector space and the identifications induced via conjugation by $GL(m,\C)$ are respected.  By denoting

\[
\hat j_k:=\hat j_k(0) \sqcup \hat j_k(1) \sqcup \ldots \sqcup \hat j_{k}(k+1),
\]
we get the following commutative diagram

\[
\xymatrix{
H^3(\mathcal{B}_\textup{alt}^\infty(\mathfrak{S}^\omega_\bullet)) \ar[r]^{S^3_\omega(n)} \ar[d]_{H^3(\hat j_\bullet^*)} & H^3_b(SL^\delta(n,\C_\omega)) \ar[d]^{j^*_b}\\
H^3(\mathcal{B}_\textup{alt}^\infty(\mathfrak{S}_\bullet)) \ar[r]^{S^3(n)} & H^3_b(SL^\delta(n,\C)),\\
}
\]
where $\hat j^*_\bullet$ are the maps induced by $\hat j_\bullet$ on the Borel cochains. We will prove that $\textup{Vol}=\textup{Vol}^\omega \circ \hat j_3$, that is $H^3(\hat j^*_\bullet)[\textup{Vol}^\omega]=[\textup{Vol}]$. Let $m \in \{0,\ldots, 4\}$. It is clear that $\textup{Vol}=\textup{Vol}^\omega \circ \hat j_3(m)$ for $m \neq 2$ because both sides are equal to zero. Let now consider $[\C^2;(v^0,\ldots,v^3)] \in \mathfrak{S}_3(2)$. If any of these vectors is $0$ both functions evaluated on the $4$-tuple give us back $0$. Hence, we can suppose that each $v^i$ is different from $0$. If the vectors $v^0,\ldots,v^3$ are in general position into $\C^2$, they still remain in general position into $\C^2_\omega$. Thus

\begin{align*}
& \textup{Vol}^\omega \circ \hat j_3(2)[\C^2;(v^0,\ldots,v^3)]=\textup{Vol}^\omega[\C^2_\omega;(v^0,\ldots,v^3)]=\text{$\omega$-$\lim_{l \to \infty}$} \textup{Vol}(v^0,\ldots,v^3)\\
&=\textup{Vol}(v^0,\ldots,v^3)=\textup{Vol}[\C^2;(v^0,\ldots,v^3)].
\end{align*}

In the same way if $(v^0,\ldots,v^3)$ are not in general position into $\C^2$, they will not be in general position into $\C^2_\omega$ either, so both $\textup{Vol}^\omega \circ \hat j_3(2)$ and $\textup{Vol}$ will evaluate to be zero, as desired. 
\end{proof}

We want now to express $\beta^\omega_n(\rho_\omega)$ in terms of boundary maps. Recall that the complement of $N$ is $M$ is given by a finite union $\bigcup_{i=1}^h C_i$ of cuspidal neighborhoods. For every $i=1,\ldots,h$ the fundamental group $\pi_1(C_i)=H_i$ is an abelian parabolic subgroup of $PSL(2,\C)$, hence it has a unique fixed point $\xi_i$ in $\mathbb{P}^1(\C)$. We define the set

\[
\mathscr{C}(\Gamma):=\bigcup_{i=1}^h \Gamma.\xi_i.
\]

\begin{deft}
If $\Gamma = \pi_1(M)$ as above, given a representation $\rho_\omega: \Gamma \rightarrow SL(n,\C_\omega)$, a \textit{decoration} for $\rho_\omega$ is a map

\[
\varphi_\omega: \mathscr{C}(\Gamma) \rightarrow \mathscr{F}(n,\C_\omega)
\]
that is equivariant with respect to $\rho_\omega$.
\end{deft}

Recall now that the cocycle $B^\omega_n$ is a strict cocycle, as in the standard case. Hence the class $(c \circ (i^*_b)^{-1} \circ (\rho_\omega)^*_b)\beta^\omega(n)$ can be represented in $H^3_b(\Gamma)$ by $\varphi_\omega^*(B^\omega_n)$, where $\varphi_\omega$ is a decoration for $\rho_\omega$ (we refer to~\cite[Corollary 2.7]{burger:articolo} for this result about the pullback of strict cocycles along boundary maps). 
In order to realize the corresponding cocycle in $H^3_b(N,\partial N)$, we identify the universal cover $\tilde N$ of $N$ with $\mathbb{H}^3$ minus a set of $\Gamma$-equivariant horoballs, each one centered at an element $\xi \in \mathscr{C}(\Gamma)$. We define a map $p:\tilde N \rightarrow \mathscr{C}(\Gamma)$ in two steps. We first send each horospherical section to the corresponding element. Then, for the interior of $\tilde N$, we map a fundamental domain to a choosen $\xi_0 \in \mathscr{C}(\Gamma)$ and we extend equivariantly. In this way, any bounded $\Gamma$-invariant cocycle $c:\mathscr{C}(\Gamma) \rightarrow \R$ determines a relative cocycle on $(N, \partial N)$ as it follows

\[
\{ \sigma: \Delta^3 \rightarrow \tilde N \} \mapsto c(p(\sigma(e_0)),\ldots,p(\sigma(e_3))).
\]

If $\tau$ is a relative triangulation of $(N, \partial N)$ and $\tilde \tau$ is the lifted triangulation of a fundamental domain in $(\tilde N,\partial \tilde N)$,  the $\omega$-Borel invariant $\beta^\omega_n(\rho_\omega)$ can be computed by the following formula

\[
\beta^\omega_n(\rho_\omega)=\sum_{\tilde \sigma \in \tilde \tau} B^\omega_n(\varphi_\omega(p(\tilde \sigma(e_0))),\varphi_\omega(p(\tilde \sigma(e_1))),\varphi_\omega(p(\tilde \sigma(e_2))),\varphi_\omega(p(\tilde \sigma(e_3))))
\]
where $\tilde \sigma$ is a lifted copy of the simplex $\sigma \in \tau$.

\section{The case $n=2$ and properties of the invariant $\beta^\omega_2(\rho_\omega)$}

In this section we are going to focus our attention on the case of representations into $SL(2,\C_\omega)$. Suppose to have a sequence of representations $\rho_l: \Gamma \rightarrow SL(2,\C)$ that determines a representation $\rho_\omega:\Gamma \rightarrow SL(2,\C_\omega)$. A sequence of decorations $\varphi_l$ for $\rho_l$ produces in a natural way a decoration $\varphi_\omega$.  Indeed it suffices to compose the standard projection $\pi:\mathbb{P}^1(\C)^\N \rightarrow \mathbb{P}^1(\C)_\omega \cong \mathbb{P}^1(\C_\omega)$ with the product map $\prod \varphi_l:\mathbb{P}^1(\C) \rightarrow \mathbb{P}^1(\C)^\N$. We say that a decoration is \textit{non-degenerate} if for every $\xi_0,\ldots,\xi_3 \in \mathscr{C}(\Gamma)$ we have that the 4-tuple $(\varphi_\omega(\xi_0),\ldots,\varphi_\omega(\xi_3))$ contains at least 3 distinct points. If the decoration $\varphi_\omega$ is non-degenerate we have

\begin{align*}
\beta^\omega_2(\rho_\omega)&=\sum_{\tilde \sigma \in \tilde \tau} B^\omega_2(\varphi_\omega(p(\tilde \sigma(e_0))),\varphi_\omega(p(\tilde \sigma(e_1))),\varphi_\omega(p(\tilde \sigma(e_2))),\varphi_\omega(p(\tilde \sigma(e_3))))\\
&=\text{$\omega$-$\lim_{l \to \infty}$} \sum_{\tilde \sigma \in \tilde \tau}  B_2(\varphi_l(p(\tilde \sigma(e_0))),\varphi_l(p(\tilde \sigma(e_1))),\varphi_l(p(\tilde \sigma(e_2))),\varphi_l(p(\tilde \sigma(e_3))))\\
&=\text{$\omega$-$\lim_{l \to \infty}$} \beta_2(\rho_l),
\end{align*}
where the last equality is obtained by applying Corollary 2.7 of~\cite{burger:articolo}. The third equality exploits the non-degenerancy of the decoration $\varphi_\omega$. Hence we get 

\begin{prop}\label{limit}
Let $\rho_l:\Gamma \rightarrow SL(2,\C)$ be a sequence of representations with decorations $\varphi_l$. Let $\rho_\omega: \Gamma \rightarrow SL(2,\C_\omega)$ be the representation associated to the sequence $\rho_l$. If the decoration $\varphi_\omega$ produced by the sequence $\varphi_l$ is non-degenerate, we have

\[
\beta^\omega_2(\rho_\omega)=\text{$\omega$-$\lim_{l \to \infty}$} \beta_2(\rho_l).
\] 
\end{prop}

\begin{cor}
Let $\rho_l:\Gamma \rightarrow SL(2,\C)$ be a sequence of representations with decorations $\varphi_l$. Let $\rho_\omega: \Gamma \rightarrow SL(2,\C_\omega)$ be the representation associated to the sequence $\rho_l$. Suppose $\beta^\omega_2(\rho_\omega)=\textup{Vol}(M)$. If the decoration $\varphi_\omega$ produced by the sequence $\varphi_l$ is non-degenerate, there must exist a sequence $g_l \in SL(2,\C)$ and a representation $\rho_\infty:\Gamma \rightarrow SL(2,\C)$ such that 

\[
\textup{$\omega$-}\lim_{l \to \infty} g_l \rho_l(\gamma) g_l^{-1}=\rho_\infty(\gamma).
\]
\end{cor}

\begin{proof}
Thanks to the assumption of non-degenerancy, by applying Proposition~\ref{limit} we desume that $\omega$-$\lim_{\l \to \infty} \beta_2(\rho_l)=\textup{Vol}(M)$. The statement now follows directly by~\cite[Theorem 1.1]{savini:articolo}.
\end{proof}

\begin{oss}
The representation $\rho_\infty$ which appears in the previous corollary as limit of the sequence $\rho_l$ has to be a lift of the standard lattice embedding $i:\Gamma \rightarrow PSL(2,\C)$.
\end{oss} 

Assume that a sequence of representations $\rho_l:\Gamma \rightarrow SL(2,\C)$ diverges to a ideal point of the character variety $X(\Gamma,SL(2,\C))$ and let $\rho_\omega:\Gamma \rightarrow SL(2,\C_\omega)$ be the representation associated to the sequence. Recall that the identification between $SL(2,\C_\omega)$ and $SL(2,\C)_\omega$ implies that the representation $\rho_\omega$ produces in a natural way an isometric action of $\Gamma$ on the asymptotic cone $C_\omega(\mathbb{H}^3,d/\lambda_l,O)$. We are going to restrict our attention to reducible actions with non-trivial length function. We first recall the following

\begin{deft}
Let $\mathcal{T}$ be a real tree on which $\Gamma$ acts via isometries. We say that the action is \textit{reducible} if one of the following holds:

\begin{itemize}
	\item The action of $\Gamma$ admits a global fixed point.
	\item There exists an end $\varepsilon \in \partial_\infty \mathcal{T}$ fixed by $\Gamma$.
	\item There exists a $\Gamma$-invariant line $L \subset \mathcal{T}$. 
\end{itemize}
\end{deft}

\begin{prop}\label{reducible}
Let $\rho_l:\Gamma \rightarrow SL(2,\C)$ be a sequence of representations and suppose it determines a representation $\rho_\omega:\Gamma \rightarrow SL(2,\C_\omega)$ such that the isometric action induced by $\rho_\omega$ on $C_\omega(\mathbb{H}^3,d/\lambda_l,O)$ has non-trivial length function. If the action is reducible then $\beta_2^\omega(\rho_\omega)=0$.
\end{prop}

\begin{proof}
Since the length function associated to the action induced by $\rho_\omega$ is non-trivial then the action does not admit a global  fixed point. Moreover, since the action is reducible, it must admit either a fixed end or an invariant line. 
Suppose that there exists an end fixed by $\Gamma$. By~\cite[Proposition 3.20]{parreau:articolo} the asymptotic cone $C_\omega(\mathbb{H}^3,d/\lambda_l,O)$ is naturally identified with the Bass--Serre tree $\Delta^{BS}(SL(2,\C_\omega))$ associated to $SL(2,\C_\omega)$. Hence, there must exist an end of $\Delta^{BS}(SL(2,\C_\omega))$ fixed by the representation $\rho_\omega$. Thus the image $\rho_\omega(\Gamma)$ is a subgroup of a suitable Borel subgroup $N_\omega$ of $SL(2,\C_\omega)$ and hence it is solvable, so amenable by~\cite[Corollary 4.1.7]{zimmer:libro}. This implies that the map $(\rho_\omega)^*_b=0$ from which we conclude $\beta^\omega_2(\rho_\omega)=0$. \\
\indent Suppose now that the action of $\Gamma$ admits an invariant line. This time the image $\rho_\omega(\Gamma)$ will be isomorphic to a subgroup of $\textup{Isom}(\R)$. Being $\textup{Isom}(\R)$ the semidirect group of the two amenable groups $\Z/2\Z$ and $\R$, it will be amenable by~\cite[Proposition 4.1.6]{zimmer:libro}. As before we will have $(\rho_\omega)^*_b=0$, hence $\beta^\omega_2(\rho_\omega)=0$.
\end{proof}

\begin{oss}
Another way to prove Proposition~\ref{reducible} is by using decorations. Indeed, if the action determined by $\rho_\omega$ admits a fixed end $\varepsilon_\omega \in \partial_\infty \Delta^{BS}(SL(2,\C_\omega))$ and since the boundary at infinity can be identified with $\mathbb{P}^1(\C_\omega)$, then the map $\varphi_\omega(\xi)=\varepsilon_\omega$ for $\xi \in \mathscr{C}(\Gamma)$ is a decoration and trivially it results $\beta_2^\omega(\rho_\omega)=0$.\\
\indent In the same way if the action admits an invariant line $L_\omega$, we denote by $\varepsilon_\omega^1$ and $\varepsilon_\omega^2$ the ends of the line $L_\omega$. For every $\xi \in \mathscr{C}(\Gamma)$ we can choose either $\varepsilon_\omega^1$ or $\varepsilon_\omega^2$ as the image of $\xi$ for the decoration $\varphi_\omega$. This implies that every possible choice produces a decoration for $\rho_\omega$ such that it results $\beta^\omega_2(\rho_\omega)=0$.
\end{oss}

Let $S=\{ \gamma_1,\ldots,\gamma_s \}$ be a generating set for the group $\Gamma$. Recall that if a sequence of representations $\rho_l:\Gamma \rightarrow SL(2,\C)$ diverges in the character variety $X(\Gamma,SL(2,\C))$ to an ideal point of the Morgan--Shalen compactification, then the real sequence

\[
\lambda_l:=\inf_{x \in \mathbb{H}^3} \sqrt{ \sum_{i=1}^s d(\rho_l(\gamma_i)x,x)} 
\]
is positive and divergent. As written in~\cite[Theorem 5.2]{parreau:articolo}, for any non-principal ultrafilter $\omega$ on $\N$, by fixing $(\lambda_l)_{l \in N}$ as scaling sequence, we can construct in a natural way a representation $\rho_\omega:\Gamma \rightarrow SL(2,\C_\omega)$ via the representations $\rho_l$.  

\begin{cor}
Let $\rho_l:\Gamma \rightarrow SL(2,\C)$ be a sequence of representations diverging to an ideal point of the Morgan--Shalen compactification of the character variety $X(\Gamma,SL(2,\C))$. Let $\rho_\omega:\Gamma \rightarrow SL(2,\C_\omega)$ be the natural representation determined by the sequence $(\rho_l)_{l \in \N}$. If the representation is reducible, then $\beta^\omega_2(\rho_\omega)=0$. 
\end{cor}

\begin{proof}
It follows directly from Proposition~\ref{reducible} by obsverving that the $\rho_\omega$ has non-trivial length function since it is associated to diverging sequence of representations. 
\end{proof}



\addcontentsline{toc}{chapter}{\bibname} 	

\vspace{20pt}
Alessio Savini\\
Department of Mathematics,\\
University of Bologna,\\
Piazza di Porta San Donato 5,\\
40126 Bologna,\\
Italy\\
alessio.savini5@unibo.it\\

\end{document}